\documentclass[11pt, reqno]{amsart}
\RequirePackage{amsmath,amsthm,amssymb,amsfonts, mathtools}
\RequirePackage[colorlinks,citecolor=blue,urlcolor=blue]{hyperref}
\usepackage[numbers, sort&compress]{natbib} 
\usepackage{float}
\usepackage{dsfont}

\newcommand{\ee}{\mathbb{E}}

\newcommand{\mse}{\text{MSE}}

\newtheorem{theorem}{Theorem}[section]

\newtheorem{lemma}[theorem]{Lemma}

\newtheorem{definition}[theorem]{Definition}

\begin{document}
\title[Estimation of Noise Variance and Matrix Estimation]{Adaptive Estimation of Noise Variance and Matrix Estimation via USVT Algorithm}
\author{Mona Azadkia}

\address{\newline Department of Statistics \newline Stanford University\newline Sequoia Hall, 390 Jane Stanford Way \newline Stanford, CA 94305\newline \newline \textup{\tt mazadkia@stanford.edu}} 

\maketitle
\begin{abstract}
We propose a method for estimating the entries of a large noisy matrix when the variance of the noise, $\sigma^2$, is unknown without putting any assumption on the rank of the matrix. We consider the estimator for $\sigma$ introduced by Gavish and Donoho \cite{Gavish} and give an upper bound on its mean squared error. Then with the estimate of the variance, we use a modified version of the Universal Singular Value Thresholding (USVT) algorithm introduced by Chatterjee \cite{Chatterjee} to estimate the noisy matrix. Finally, we give an upper bound on the mean squared error of the estimated matrix. 
\end{abstract}

\section{Introduction} 
Consider the statistical estimation problem where the unknown parameters of interest are the entries of a $m\times n$ matrix $M$, where $m\leq n$. Suppose we have observed the $m\times n$ matrix $X$, a noisy version of $M$, that is $X = M + \sigma A$. The entries of $A$ are i.i.d. with mean zero and variance one and $\sigma$ is an unknown positive constant. The goal is to recover $M$ from this noisy observation $X$ and give an upper bound on the mean squared error ($\mse$) of its estimate. Given an estimator $\hat{M}$ for $M$, the $\mse(\hat{M})$ is defined as
\begin{eqnarray}
\mse(\hat{M}) := \ee\Big[\frac{1}{mn}\sum_{i=1}^m\sum_{j=1}^n(m_{ij} - \hat{m}_{ij})^2\Big],
\end{eqnarray}
where $m_{ij}$ and $\hat{m}_{ij}$ are respectively the $(i, j)$th entries of the matrices $M$ and $\hat{M}$. 

The problem of estimating the entries of a large matrix from noisy and/or incomplete observations has been studied widely. A common approach for solving this problem is to assume that $M$ is a low-rank matrix. Under the low-rank assumption there is a huge body of work using spectral methods, such as \cite{Azar, Achlioptas, Keshavan1, Keshavan2, Gavish, Mazumder-Hastie-Tibshirani}. Some other works under certain model assumptions are \cite{Fazel, Rodhe-Tsybakov}.

In a different direction, Emmanuel Cand\`{e}s and his collaborators \cite{Candes-Recht, Candes-Plan, Candes-Tao, Candes-Shen-Cai} studied this problem by penalizing the nuclear norm of the matrix under convex constraints. Some other notable examples of this penalization approach are \cite{Gavish-Donoho, Koltchinskii, Koltchinskii-Lounici-Tsybakov, Negahban-Wainwright, Rennie}.

In 2015, Chatterjee \cite{Chatterjee} proposed a simple estimation procedure, the USVT algorithm. In that work, Chatterjee considered the problem of matrix estimation without putting any assumption on the rank of the matrix $M$. However, in \cite{Chatterjee} he assumed that the noise entries and therefore the entries of $X$ lie in a bounded interval. Although he added that the results should stay valid when the entries of $X-M$ are $N(0, \sigma^2)$ when $\sigma$ is known, but the problem of estimating $M$ when $\sigma$ is unknown, remained unsolved. 

Later, in \cite{Gavish} Gavish and Donoho proposed an estimator for $\sigma$ based on the observed matrix $X$. In that work, using the Mar\v{c}henko-Pastur law, they construct their estimator as a function of the median singular value of matrix $X$. They showed that this estimate almost surely converges to $\sigma$ as $n$ goes to infinity under this assumption that the additive noise is orthogonally invariant. In a different work, Nadler and Kritchman \cite{Nadler}, proposed an iterative algorithm for estimating the unknown $\sigma$. 

In this paper, we consider the problem of matrix estimation when the variance of the noise, $\sigma$, is unknown. We do not put any assumption on the rank of the matrix or boundedness of its entries. However, we consider the mild assumption that entries of the noise matrix $A$ are i.i.d. and sub-Gaussian. First, we give an upper bound on the mean squared error of the estimator of $\sigma$ from \cite{Gavish}, $\hat{\sigma}$. Using the estimate of $\sigma$, we modify the USVT algorithm \cite{Chatterjee} and give an estimate of $M$, $\hat{M}$, based on the observed matrix $X$. At last we give an upper bound on the mean squared error of $\hat{M}$. 

In section \ref{Set up and Main Results} we study the estimator of the variance. Theorem~\ref{MSESigmathm} gives an upper bound on the mean squared error of this estimate. Then we give an estimator of $M$ and in Theorem~\ref{MSEMthm} give an upper bound on the mean squared error of this estimator of $M$. In section~\ref{Simulation} we study a simulated example. Proofs of all theorems and lemmas are in section~\ref{Proofs}.

\section{Set up and Main Results} \label{Set up and Main Results}
Consider the $m \times n$ random matrix $X = [x_{ij}]$. Without loss of generality, we let $m\leq n$. Let $X$ be a noisy version of $M$, $X = M + \sigma A$, where $M = [m_{ij}]$ is a deterministic unknown matrix and $A = [\epsilon_{ij}]$ is a random matrix with  i.i.d. entries independent of $M$ and $\ee[\epsilon_{ij}] = 0$ and $\ee[\epsilon_{ij}^2] = 1$. Noise level $\sigma > 0$ is an unknown deterministic constant. The final goal is to estimate the entries of $M = [m_{ij}]\in\mathbb{R}^{m\times n}$. 

Let's assume that $\sigma$ is known and $|\epsilon_{ij}|\leq 1$, then following the steps of USVT algorithm \cite{Chatterjee} bellow, we find $\hat{M}$ an estimate of $M$.
{\centering
\begin{enumerate}
\item Let $\sum_{i=1}^m \lambda_i(X) u_i v_i^T$  be the singular value decomposition of $X$, where $\lambda_1(X)\geq\ldots\geq\lambda_m(X)$ are the singular values of $X$ and $u_i$ and $v_i$ are its left and right singular vectors.

\item Choose a small positive number $\eta\in (0,1]$ and let $S$ be the set of ``thresholded singular values",
\[
S_{\sigma} := \{i: \lambda_i(X) \geq (2+\eta)\sigma\sqrt{n}\}.
\]

\item Define $\hat{M} := \sum_{i\in S_{\sigma}} \lambda_i(X) u_i v_i^T$.
\end{enumerate}
}
Note that set $S_{\sigma}$ depends on $\sigma$ and also uses the earlier assumption on boundedness of $\epsilon_{ij}$'s.

To continue, first we introduce some notations and definitions.

\subsection{Definitions and Notation}
Let $X = [x_{ij}]$ be $m\times n$, where $m \leq n$, a random matrix with $x_{ij}$'s i.i.d. and $\ee[x_{ij}] = 0$ and $\ee[x_{ij}^2] = 1$. Let
\[
Y_{n} = \frac {1}{n}XX^{T},
\] 
and let $\lambda_1(Y_n)\geq \cdots\geq \lambda_m(Y_n)$ be the singular values of $Y_{n}$. Define $\mu _{m}$ as a random counting measure,
\[
\mu _{m}(A) = \frac{1}{m}\#\left\{\lambda _{j}(Y_n)\in A\right\},\quad A\subset \mathbb{R}.
\]

\textbf{Mar\v{c}henko-Pastur Law}: Let $m,\,n\,\to \,\infty$ such that $m/n\,\to \,\gamma \in (0,1)$ . Then $\mu _{m}\,\to \,\mu_{\text{MP}}^\gamma$ in distribution, where

\[
d\mu_{\text{MP}}^\gamma(x)={\frac {1}{2\pi}}{\frac {\sqrt {(\gamma _{+}-x)(x-\gamma _{-})}}{\gamma x}}\,\mathbf {1} _{[\gamma _{-},\gamma _{+}]}\,dx 
\]
with $\gamma _{\pm }=(1\pm {\sqrt {\gamma}})^{2}$.

Let $\mathbb{F}_\gamma(x)$ and $\mathbb{F}_n$ be the cumulative distribution function (cdf) of $\mu_{\text{MP}}^\gamma$ and $\mu _{m}$ respectively
\begin{eqnarray*}
\mathbb{F}_\gamma(x) = \int_{\gamma_{-}}^xd\mu_{\text{MP}}^\gamma(t), \qquad \mathbb{F}_n(x) = \int_{0}^x\,d\mu_{m}(t).
\end{eqnarray*}

\begin{definition} 
For the $m\times n$ matrix $X$ consider the singular value decomposition $X = U\Sigma V^T$, where $\Sigma$ is a $m\times n$ diagonal matrix with diagonal entries $\lambda_1(X) \geq\cdots\geq \lambda_m(X)\geq 0$, and $U$ and $V$ are unitary matrices of size $m\times m$ and $n\times n$ respectively. The nuclear norm of $X$ is defined as the sum of its singular values,
\begin{eqnarray*}
\|X\|_* := \sum_{i=1}^m\lambda_{i}(X).
\end{eqnarray*}
\end{definition}

\subsection{Estimation of $\sigma$}
We consider the following estimator of $\sigma$ that was proposed by Donoho and Gavish \cite{Gavish},
\begin{eqnarray*}
\hat{\sigma}(X) = \frac{\text{med}(\lambda_i(X))}{\sqrt{n\mu_\gamma}}, 
\end{eqnarray*}
where $\text{med}(\lambda_i(X))$ is the median of $\lambda_i(X)$'s and
\begin{eqnarray}\label{eq:MuGamma}
\mu_\gamma = \mathbb{F}_\gamma^{-1}(\frac{1}{2}).
\end{eqnarray}

In \cite{Gavish}, Donoho and Gavish have shown that 
\begin{eqnarray*}
\lim_{n\rightarrow\infty}\hat{\sigma}(X) = \sigma \qquad\text{a.s.}
\end{eqnarray*}
They have made this conclusion under an asymptotic frame work in which they have assumed that the distribution of $A$ is \textit{orthogonally invariant}, meaning that for $R_1$ and $R_2$, $m\times m$ and $n\times n$ orthogonal matrices, $A$ and $R_1 A R_2$ have the same distribution. We consider this estimator without the extra invariance assumption on the noise distribution. Theorem~\ref{MSESigmathm} provides an upper bound on the $\mse$ of $\hat{\sigma}(X)$.

\begin{theorem} \label{MSESigmathm}
Let $X = [x_{ij}]$ be a $m\times n$ random matrix with $x_{ij}$'s distributed i.i.d. from a sub-Gaussian distribution such that $\ee(x_{ij}) = m_{ij}$ and $Var(x_{ij}) = \sigma^2$ for unknown values of $\sigma$ and $m_{ij}$'s. For  
\begin{eqnarray*}
\hat{\sigma} = \frac{\emph{\text{med}}(\lambda_i(X))}{\sqrt{n\mu_\gamma}},
\end{eqnarray*}
we have
\begin{eqnarray*}
\ee\Big[(\hat{\sigma} - \sigma)^2\Big] &\leq & \frac{2\|M\|_*^2}{\mu_\gamma n^2} + \frac{\sigma^2 C_\gamma}{\mu_\gamma n} +  \frac{\sigma^2 C_{\gamma, \epsilon}}{\mu_\gamma n^{1 - \epsilon}},
\end{eqnarray*}
where $C_{\gamma, \epsilon}$ and $C_\gamma$ are non-negative constants independent of $n$ and $\sigma$.
\end{theorem}
If $\|M\|_* = o(n)$, Theorem~\ref{MSESigmathm} implies that the $\mse(\hat{\sigma})$ 

In \cite{Nadler} Kritchman and Nadler suggested an iterative method for estimating the $\sigma$. The strength of $\hat{\sigma}$ for us is that it's easy to compute. Also its simple definition made it possible to use random matrix theory to give the upper bound on its $\mse$.

\subsection{Modified USVT estimator of $M$} 
With $\hat{\sigma}$ as the estimator of $\sigma$, we use the following modified version of the USVT algorithm to find an estimator for $M$.\\

{\centering
\begin{enumerate}
\item Let $\sum_{i=1}^m \lambda_i(X) u_i v_i^T$  be the singular value decomposition of $X$, where $\lambda_1(X)\geq\ldots\geq\lambda_m(X)$ are the singular values of $X$ and $u_i$ and $v_i$ are the left and right singular vectors.

\item Choose a small positive number $\eta\in (0,1]$ and let $S$ be the set of ``thresholded singular values" 
\[
S_{\hat{\sigma}} = \{i: \lambda_i(X) \geq (2+\eta)\hat{\sigma}\sqrt{n}\}.
\]

\item Define $\hat{M} = \sum_{i\in S_{\hat{\sigma}}} \lambda_i(X) u_i v_i^T$.
\end{enumerate}
}

Theorem~\ref{MSEMthm} gives an upper bound on the $\mse$ of $\hat{M}$. 
\begin{theorem}\label{MSEMthm}
Let $X = [x_{ij}]$ be a $m\times n$ random matrix with $x_{ij}$'s distributed i.i.d. from a sub-Gaussian distribution such that $\ee(x_{ij}) = m_{ij}$ and $Var(x_{ij}) = \sigma^2$ for unknown values of $\sigma$ and $m_{ij}$'s. Let 
\begin{eqnarray*}
\hat{\sigma} = \frac{\emph{\text{med}}(\lambda_i(X))}{\sqrt{n\mu_\gamma}},
\end{eqnarray*}
and for $\eta\in(0, 1)$ define $\hat{M}$ using the modified USVT algorithm. Then
\begin{eqnarray*}
\lefteqn{\emph{\mse}(\hat{M}) = \ee\big[\frac{1}{mn}\|\hat{M} - M\|_F^2\big]} \\ 
&\leq &  \frac{C_0\sigma\|M\|_*}{\gamma n\sqrt{n}} + (2 + \eta)^2\sigma^2\left(C_1 + C_2\frac{\|M\|_*^4}{\sigma^4\mu_\gamma^2 n^4}\right)^{1/2}\max\{\sqrt{\frac{C_{\epsilon, \gamma}\sigma}{\mu_\gamma\eta^2 n^{1 - \epsilon}}}, \sqrt{\frac{\|M\|_*^2}{n^2\sigma^2\eta^2\mu_\gamma}}\},
\end{eqnarray*}
where $C_0, C_1$, and $C_2$ are constants, and $C_{\epsilon, \gamma}$ is a constant that depends on $\epsilon$ and $\gamma$. 
\end{theorem}
The upper bounds depend on the $\|M\|_*$ and $\sigma$, the unknown parameters of the model, and the parameter of choice $\eta$. Theorem~\ref{MSEMthm} implies that if $\|M\|_* = o(n)$ then $\emph{\mse}(\hat{M}) $ is of order of $\sigma^2$. 
\section{Simulation} \label{Simulation}
In this section we consider a simple simulated example. Let $n = 1000$ and $m = 200$ and we consider the sequence $\lambda_1\geq\cdots\geq\lambda_{200}$ where $\lambda_i = \exp(3 - i/50)$ for $i\in\{1, \cdots, 200\}$. Then for each $r\in\{50, 100, 150, 200\}$ we define the following signal matrix
\begin{eqnarray*}
M_r = UD_rV^T
\end{eqnarray*}
where $D_r$ is a $m\times n$ rectangular diagonal matrix with the $i$-th diagonal entry equal to $\lambda_i \mathds{1}\{i \leq r\}$, and $U$ and $V$ are randomly uniform orthogonal matrices of size $m\times m$ and $n\times n$ respectively. For each choice of $r\in\{50, 100, 150, 200\}$, we generated 100, independent noise matrix $A$ with $N(0, 1)$ i.i.d. entries and considered observed matrix $X = M_r + \sigma A$ for different values of $\sigma$. In Figure~\ref{mseCombined} we have plotted the mean squared error of $\hat{\sigma}$ and $\hat{M_r}$ for different values of $r$ and $\sigma$. We have set $\eta = 0.02$ in the USVT algorithm.

\begin{figure}
  \centering
    \includegraphics[width=1\textwidth]{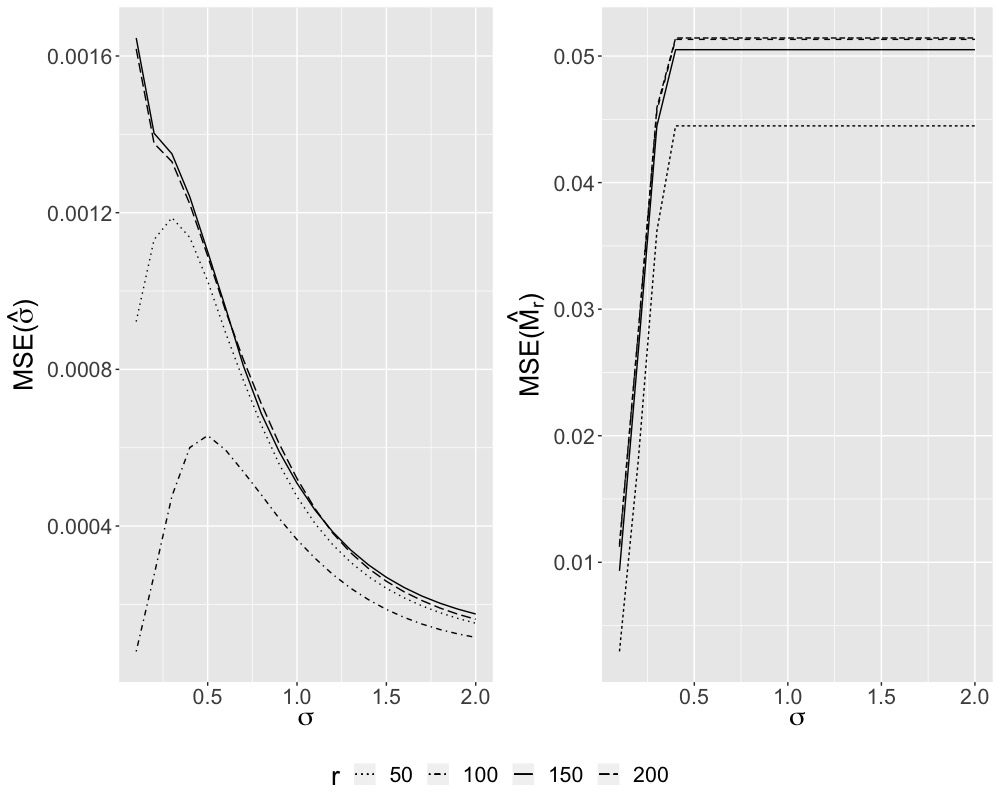}
    \caption{Estimated mean squared error of $\hat{\sigma}$ and $\hat{M}$ for different values of rank, $r = 50, 100, 150$, and $200$.  The our error bounds in Theorems~\ref{MSESigmathm} and~\ref{MSEMthm} do not depend on the rank, but it depends on the nuclear norm of the mean matrix $M_r$. In this example we have $\|M_{50}\|_* = 641.193$, $\|M_{100}\|_* = 877.0753$, $\|M_{150}\|_* = 963.851$, and $\|M_{200}\|_* = 995.775$.}\label{mseCombined}
\end{figure}

\section{Proofs} \label{Proofs}
In what follows $c$, $C$ and $C_i$ for $i\in\mathbb{N}$ are non-negative constants independent of the parameters of the problem, and $C_\gamma$ and $C_{\epsilon, \gamma}$ are non-negative constants that only depends on $\epsilon$ and pair of $\epsilon$ and $\gamma$ respectively. For simplicity these values may change from line to line or even in a line. For simplicity in notation and without loss of generality we assume that $m$ is an even number and therefore use $\lambda_{m/2}$ instead of $\text{med}(\lambda_i)$ for the median singular value. 
\subsection{Proofs of Theorems~\ref{MSESigmathm} and \ref{MSEMthm}} \label{ProofMain}
\subsubsection{Proof of Theorem~\ref{MSESigmathm}}
\begin{proof}
By adding and subtracting $\lambda_{m/2}(X - M)$ and inequality $(a + b)^2\leq 2a^2 + 2b^2$ we have
\begin{eqnarray}
\ee\Big[\Big(\hat{\sigma} - \sigma\Big)^2\Big] & = & \ee\Big[\Big(\frac{\lambda_{m/2}(X)}{\sqrt{n\mu_\gamma}} - \sigma\Big)^2\Big] \nonumber\\
&\leq & \frac{2}{n\mu_\gamma}\ee\Big[\Big(\lambda_{m/2}(X) - \lambda_{m/2}(X - M)\Big)^2\Big]  \label{distLambda}\\ 
&+& \frac{2\sigma^2}{n\mu_\gamma}\ee\Big[\Big(\lambda_{m/2}(\frac{X - M}{\sigma}) - \sqrt{n\mu_\gamma}\Big)^2\Big]. \label{distLambdaMean}
\end{eqnarray}

Using Lemma~\ref{diffFsqrt} for~\eqref{distLambdaMean} we have,
\begin{eqnarray}
\frac{2\sigma^2}{n\mu_\gamma}\ee\Big[\Big(\lambda_{m/2}(\frac{X - M}{\sigma}) - \sqrt{n\mu_\gamma}\Big)^2\Big] & = & \frac{2\sigma^2}{\mu_\gamma}\ee\Big[\Big(\lambda_{m/2}(\frac{A}{\sqrt{n}}) - \sqrt{\mu_\gamma}\Big)^2\Big]\nonumber\\
& = & \frac{2\sigma^2}{\mu_\gamma}\ee\Big[\Big( \sqrt{\mathbb{F}_n^{-1}(\frac{1}{2})}- \sqrt{\mathbb{F}_\gamma^{-1}(\frac{1}{2})}\Big)^2\Big] \nonumber\\
& \leq &  \frac{\sigma^2 C_{\epsilon,\gamma}}{\mu_\gamma n^{1 - \epsilon}}. 
\end{eqnarray}

To give an upper bound on~\eqref{distLambda}, we use the following inequality from \cite{Bhatia97} (page 75). For any two $n\times n$ matrices $A_1$ and $A_2$ and any two indices $i$ and $j$ such that $i + j\leq n + 1$, 
\begin{eqnarray}\label{bhatia}
\lambda_{i + j - 1}(A_1 + A_2) \leq \lambda_i(A_1) + \lambda_j(A_2).
\end{eqnarray}

For matrices $X - M$ and $M$, $0 < k \leq \lfloor\sqrt{n}\rfloor$, and indices $m/2 - k$ and $k + 1$,~\eqref{bhatia} gives
\begin{eqnarray}
\lambda_{m/2}(X) & \leq & \lambda_{m/2 - k}(X - M) + \lambda_{k+1}(M), \label{first}
\end{eqnarray}
and for matrices $X$ and $M$, and indices $m/2$ and $k + 1$,~\eqref{bhatia} gives
\begin{eqnarray}
\lambda_{m/2 + k}(X - M) & \leq & \lambda_{m/2}(X) + \lambda_{k+1}(M).\label{second}
\end{eqnarray}

Subtracting $\lambda_{m/2}(X - M)$ from the left and right sides of~\eqref{first} and~\eqref{second} gives
\begin{eqnarray*}
\lambda_{m/2}(X) - \lambda_{m/2}(X - M) & \leq & \lambda_{m/2 - k}(X - M) - \lambda_{m/2}(X - M) + \lambda_{k+1}(M),\\ 
\lambda_{m/2}(X) -  \lambda_{m/2}(X - M) &\geq & \lambda_{m/2+k}(X - M) -  \lambda_{m/2}(X - M) - \lambda_{k+1}(M).
\end{eqnarray*}
Therefore 
\begin{eqnarray}\label{boundDiffLa}
\lefteqn{(\lambda_{m/2}(X) - \lambda_{m/2}(X - M))^2} \nonumber\\
&\leq & 2\max\{(\lambda_{m/2 - k}(X - M) - \lambda_{m/2}(X - M))^2, (\lambda_{m/2 + k}(X - M) - \lambda_{m/2}(X - M))^2\} \nonumber\\
&+&2\lambda_{k+1}(M)^2
\end{eqnarray}
Note that
\begin{eqnarray}\label{upLambdaK}
\lambda_{k+1}(M) \leq \frac{\|M\|_*}{k+1}.
\end{eqnarray}
Taking expectation of~\eqref{boundDiffLa} and inequality~\eqref{upLambdaK} gives
\begin{eqnarray}
\lefteqn{\ee[(\lambda_{m/2}(X) - \lambda_{m/2}(X - M))^2]} \nonumber\\
&\leq & 2\ee[(\lambda_{m/2 - k}(X - M) - \lambda_{m/2}(X - M))^2] \label{minusK}\\
&+& 2\ee[(\lambda_{m/2 + k}(X - M) - \lambda_{m/2}(X - M))^2] \label{plusK}\\
&+& 2(\frac{\|M\|_*}{k+1})^2. \nonumber
\end{eqnarray}
To give an upper bound on~\eqref{minusK} we use the following decomposition

\begin{eqnarray}
\lambda_{m/2 - k}(X - M) - \lambda_{m/2}(X - M) &=& \sigma\sqrt{n}(\sqrt{\mathbb{F}_n^{-1}(\frac{1}{2} - \frac{k}{n})} - \sqrt{\mathbb{F}_n^{-1}(\frac{1}{2})}) \nonumber\\
&=& \sigma\sqrt{n}\left[\sqrt{\mathbb{F}_n^{-1}(\frac{1}{2} - \frac{k}{n})} - \sqrt{\mathbb{F}_\gamma^{-1}(\frac{1}{2} - \frac{k}{n})}\right] +  \nonumber\\
&& \sigma\sqrt{n}\left[\sqrt{\mathbb{F}_\gamma^{-1}(\frac{1}{2})} - \sqrt{\mathbb{F}_n^{-1}(\frac{1}{2})}\right] +  \nonumber\\
&& \sigma\sqrt{n}\left[\sqrt{\mathbb{F}_\gamma^{-1}(\frac{1}{2} - \frac{k}{n})} - \sqrt{\mathbb{F}_\gamma^{-1}(\frac{1}{2})}\right]. 
\label{usefulDecomposition}
\end{eqnarray}

Using inequality $(a + b)^2 \leq 2a^2 + 2b^2$, and Lemma~\ref{diffFsqrt} we have
\begin{eqnarray}
\lefteqn{\ee[(\lambda_{m/2 - k}(X - M) - \lambda_{m/2}(X - M))^2] \leq }\nonumber\\
& & 4\sigma^2 n\ee\left[\left(\sqrt{\mathbb{F}_n^{-1}(\frac{1}{2} - \frac{k}{n})} - \sqrt{\mathbb{F}_\gamma^{-1}(\frac{1}{2} - \frac{k}{n})}\right)^2\right] + \nonumber\\
&& 4\sigma^2 n\ee\left[\left(\sqrt{\mathbb{F}_\gamma^{-1}(\frac{1}{2})} - \sqrt{\mathbb{F}_n^{-1}(\frac{1}{2})}\right)^2\right] + \nonumber\\
&& 2\sigma^2 n\left[\sqrt{\mathbb{F}_\gamma^{-1}(\frac{1}{2} - \frac{k}{n})} - \sqrt{\mathbb{F}_\gamma^{-1}(\frac{1}{2})}\right]^2  \nonumber\\
&& \leq 2\sigma^2 C_{\epsilon,\gamma} n^\epsilon + 2\sigma^2 n\left[\sqrt{\mathbb{F}_\gamma^{-1}(\frac{1}{2} - \frac{k}{n})} - \sqrt{\mathbb{F}_\gamma^{-1}(\frac{1}{2})}\right]^2. \label{diff1}
\end{eqnarray}
Similary
\begin{eqnarray}
\lefteqn{\ee[(\lambda_{m/2 + k}(X - M) - \lambda_{m/2}(X - M))^2] \leq }\nonumber\\
& & 4\sigma^2 n\ee\left[\left(\sqrt{\mathbb{F}_n^{-1}(\frac{1}{2} + \frac{k}{n})} - \sqrt{\mathbb{F}_\gamma^{-1}(\frac{1}{2} + \frac{k}{n})}\right)^2\right] + \nonumber\\
&& 4\sigma^2 n\ee\left[\left(\sqrt{\mathbb{F}_\gamma^{-1}(\frac{1}{2})} - \sqrt{\mathbb{F}_n^{-1}(\frac{1}{2})}\right)^2\right] + \nonumber\\
&& 2\sigma^2 n\left[\sqrt{\mathbb{F}_\gamma^{-1}(\frac{1}{2} + \frac{k}{n})} - \sqrt{\mathbb{F}_\gamma^{-1}(\frac{1}{2})}\right]^2  \nonumber\\
&& \leq 2\sigma^2 C_{\epsilon,\gamma} n^\epsilon + 2\sigma^2 n\left[\sqrt{\mathbb{F}_\gamma^{-1}(\frac{1}{2} + \frac{k}{n})} - \sqrt{\mathbb{F}_\gamma^{-1}(\frac{1}{2})}\right]^2.\label{diff2}
\end{eqnarray}

Using the mean value theorem, we find the following bound on~\eqref{diff1}, and~\eqref{diff2}.
\begin{eqnarray}\label{boundFlambda}
|\sqrt{\mathbb{F}_\gamma^{-1}(\frac{1}{2} \pm \frac{k}{n})} - \sqrt{\mathbb{F}_\gamma^{-1}(\frac{1}{2})}| &\leq & C_\gamma|\mathbb{F}_\gamma^{-1}(\frac{1}{2} \pm \frac{k}{n}) - \mathbb{F}_\gamma^{-1}(\frac{1}{2})|\nonumber \\
&\leq & \frac{C_\gamma}{\sqrt{n}}.
\end{eqnarray}

Putting~\eqref{diff1}, and~\eqref{diff2} together with~\eqref{boundFlambda} gives
\begin{eqnarray}\label{last1}
\ee[(\lambda_{m/2}(X) - \lambda_{m/2}(X - M))^2] \leq \sigma^2 C_{\epsilon,\gamma} n^\epsilon + \sigma^2 C_\gamma.
\end{eqnarray}

Inequalities~\eqref{last1}, and~\eqref{upLambdaK} with $k = \lfloor\sqrt{n}\rfloor$ completes the proof,
\begin{eqnarray*}
\ee\Big[(\hat{\sigma} - \sigma)^2\Big] &\leq & \frac{2\|M\|_*^2}{\mu_\gamma n^2} + \frac{\sigma^2 C_\gamma}{\mu_\gamma n} +  \frac{\sigma^2 C_{\gamma, \epsilon}}{\mu_\gamma n^{1 - \epsilon}}.
\end{eqnarray*}

\end{proof}

\subsection{Proof of Theorem~\ref{MSEMthm}}
\begin{proof}
Consider the following two sets $E_1$ and $E_2$,
\begin{eqnarray*}
E_1 &=& \{\|X - M\| \leq (2 + \frac{\eta}{2})\sigma\sqrt{n}\}, \\
E_2 &=& \{|\hat{\sigma} - \sigma|\leq\frac{\eta}{20}\sigma \}.
\end{eqnarray*}
Consider the following decomposition,
\begin{eqnarray*}
\ee\|\hat{M} - M\|_F^2 = \ee[\|\hat{M} - M\|_F^2\mathds{1}\{(E_1\cap E_2)^\mathsf{c}\}] + \ee[\|\hat{M} - M\|_F^2\mathds{1}\{(E_1\cap E_2)\}].
\end{eqnarray*}
Using Lemma~\ref{lemmaChatterjee} from \cite{Chatterjee}, which we copy without its proof, we find an upper bound on $\ee[\|\hat{M} - M\|_F^2\mathds{1}\{(E_1\cap E_2)\}]$. 
\begin{lemma}\label{lemmaChatterjee}
Let $D = \sum_{i=1}^m\sigma_i u_i v_i^T$ be the singular value decomposition of $D$. Fix any $\delta > 0$ and define 
\begin{eqnarray*}
\hat{B} \coloneqq \sum_{i : \sigma_i > (1+\delta)\|D - B\|}\sigma_i x_i y_i^T.
\end{eqnarray*} 
Then
\begin{eqnarray*}
\|\hat{B} - B\|_F \leq K(\delta)(\|D - B\|\|B\|_*)^{1/2},
\end{eqnarray*}
where $K(\delta) = (4 + 2\delta)\sqrt{2/\delta} + \sqrt{2 + \delta}$.
\end{lemma}

For $\delta \geq \eta/5$, we have $K(\delta)\leq C\sqrt{1+\delta}$. On the set $E_1\bigcap E_2$, we have $\delta \geq \eta/5$. Thus by using Lemma~\ref{lemmaChatterjee} we have,
\begin{eqnarray*}
\|\hat{M} - M\|_F^2 &\leq & C(1 + \delta)\|X - M\|\|M\|_* \\
&\leq & C\sqrt{n}\sigma\|M\|_*.
\end{eqnarray*}
Therefore
\begin{eqnarray*}\label{intersection}
 \ee[\|\hat{M} - M\|_F^2 \mathds{1}\{(E_1\cap E_2)\}] \leq C\sqrt{n}\sigma\|M\|_*.
\end{eqnarray*}

Then by Cauchy-Schwartz inequality we get
\begin{eqnarray}
\ee[\|\hat{M} - M\|_F^2\mathds{1}\{(E_1\cap E_2)^\mathsf{c}\}] \leq \sqrt{\ee[\|\hat{M} - M\|_F^4]\ee[(E_1\cap E_2)^\mathsf{c}]}. \label{eq:secondTerm}
\end{eqnarray}
Note that
\begin{eqnarray*}
\mathbb{P}[(E_1\cap E_2)^\mathsf{c}] \leq \mathbb{P}(E_1^\mathsf{c}) +  \mathbb{P}(E_2^\mathsf{c}).
\end{eqnarray*}
By Proposition 2.4 in \cite{Rudelson-Vershynin} 
\begin{eqnarray}
\mathbb{P}(E_1^\mathsf{c}) \leq Ce^{-c\eta^2 n}, \label{eq:E1Bound}
\end{eqnarray}
and by Chebyshev's inequality 
\begin{eqnarray}
\mathbb{P}(E_2^\mathsf{c}) &\leq & \frac{800\|M\|_*^2}{\mu_\gamma \sigma^2\eta^2 n^2} + \frac{400C_\gamma}{\mu_\gamma\eta^2 n} +  \frac{400 C_{\gamma, \epsilon}}{\mu_\gamma \eta^2 n^{1 - \epsilon}}. \label{eq:E2Bound}
\end{eqnarray}
Inequalities~\eqref{eq:E1Bound} and~\eqref{eq:E2Bound} together give
\begin{eqnarray}\label{boundOnComplementUnion}
\mathbb{P}[(E_1\cap E_2)^\mathsf{c}] \leq Ce^{-c\eta^2 n} +  \frac{800\|M\|_*^2}{\mu_\gamma \sigma^2\eta^2 n^2}  +  \frac{C_{\gamma, \epsilon}}{\mu_\gamma \eta^2 n^{1 - \epsilon}}.
\end{eqnarray}

To finish finding a bound on~\eqref{eq:secondTerm}, we find an upper bound on $\ee[\|\hat{M} - M\|^4_F]$.  Using inequality $(a+b)^2\leq 2a^2 + 2b^2$, we have
\begin{eqnarray*}
\|\hat{M} - M\|_F^4 &=& (\sum_{ij}(m_{ij} - \hat{m}_{ij})^2)^2\nonumber\\
&=& (\sum_{ij}(m_{ij} - x_{ij} + x_{ij} - \hat{m}_{ij})^2)^2\nonumber\\
&\leq & (2\sum_{ij}(m_{ij} - x_{ij})^2 + (x_{ij} - \hat{m}_{ij})^2)^2\nonumber\\
&=&  4(\sum_{ij}(m_{ij} - x_{ij})^2 + \sum_{ij}(x_{ij} - \hat{m}_{ij})^2)^2\nonumber\\
&\leq & 8(\sum_{ij}(m_{ij} - x_{ij})^2)^2 + 8(\sum_{ij}(x_{ij} - \hat{m}_{ij})^2)^2\nonumber \\
&=& 8\|M - X\|_F^4 + 8\|X-\hat{M}\|_F^4.
\end{eqnarray*}
Note that
\begin{eqnarray*}
\|\hat{M} - X\|_F^4 &=& \|\sum_{\lambda_i(X) < (2+\eta)\hat{\sigma}\sqrt{n}}\lambda_i(X) u_iv_i^T\|_F^4 \nonumber\\
&=& ( \|\sum_{\lambda_i(X) < (2+\eta)\hat{\sigma}\sqrt{n}}\lambda_i(X) u_iv_i^T\|_F^2)^2 \nonumber\\
&\leq & n^4(2 + \eta)^4{\hat{\sigma}}^4.
\end{eqnarray*}
Therefore
\begin{eqnarray}
\ee[\|\hat{M} - X\|_F^4] &\leq & n^4(2 + \eta)^4\ee[{\hat{\sigma}}^4], \label{eq:MhatMinusX}
\end{eqnarray}
and Lemma~\ref{fourthMomentSigmaHat} gives
\begin{eqnarray}
\ee[\|\hat{M} - X\|_F^4] &\leq & 4n^4(2 + \eta)^4(4\sigma^4 + \frac{C_{\epsilon, \gamma}\sigma^4}{n^{2 - \epsilon}} + \frac{64\|M\|_*^4}{n^4\mu_\gamma^2})
\end{eqnarray}

Now note that 
\begin{eqnarray*}
 \|X - M\|_F^4 &=& \sum_{i, j} (x_{ij} - m_{ij})^4 + \sum_{(i, j)\neq(l, k)} (x_{ij} - m_{ij})^2 (x_{lk} - m_{lk})^2 
\end{eqnarray*}
Since $x_{ij}$'s are i.i.d sub-Gaussian random variables and thus their forth moment is bounded, 
\begin{eqnarray}
\ee[\|X - M\|_F^4] &\leq & C\gamma^2 \sigma^4 n^4. \label{eq:XMinusM}
\end{eqnarray}
Inequalities~\eqref{eq:XMinusM} and~\eqref{eq:MhatMinusX} and Lemma~\ref{fourthMomentSigmaHat} give
\begin{eqnarray*}
\lefteqn{\ee\|\hat{M} - M\|_F^2  \leq  C_0\sqrt{n}\sigma\|M\|_* + }\\
& & \left(4n^4(2 + \eta)^4(C_1\sigma^4 + \frac{C_{\epsilon, \gamma}\sigma^4}{n^{2 - \epsilon}} + \frac{64\|M\|_*^4}{n^4\mu_\gamma^2})\right)^{1/2}\times\left(C_2e^{-c\eta^2 n} +  \frac{800\|M\|_*^2}{\mu_\gamma \sigma^2\eta^2 n^2}  +  \frac{C_{\gamma, \epsilon}}{\mu_\gamma \eta^2 n^{1 - \epsilon}}\right)^{1/2}.
\end{eqnarray*}
This completes the proof of Theorem~\ref{MSEMthm},
\begin{eqnarray*}
\lefteqn{\mse(\hat{M}) = \frac{1}{\gamma n^2}\ee[\|\hat{M} - M\|_F^2]\leq} \\
& & \frac{C_0\sigma\|M\|_*}{\gamma n\sqrt{n}} + (2 + \eta)^2\sigma^2\left(C_1 + C_2\frac{\|M\|_*^4}{\sigma^4\mu_\gamma^2 n^4}\right)^{1/2}\left(\frac{C_{\epsilon, \gamma}\sigma}{\mu_\gamma\eta^2 n^{1 - \epsilon}} + \frac{\|M\|_*^2}{n^2\sigma^2\eta^2\mu_\gamma}\right)^{1/2} \leq\\
& & \frac{C_0\sigma\|M\|_*}{\gamma n\sqrt{n}} + \sqrt{2}(2 + \eta)^2\sigma^2\left(C_1 + C_2\frac{\|M\|_*^4}{\sigma^4\mu_\gamma^2 n^4}\right)^{1/2}\max\{\sqrt{\frac{C_{\epsilon, \gamma}\sigma}{\mu_\gamma\eta^2 n^{1 - \epsilon}}}, \sqrt{\frac{\|M\|_*^2}{n^2\sigma^2\eta^2\mu_\gamma}}\}. \\
\end{eqnarray*}
\end{proof}
\subsection{Proofs of the lemmas} \label{ProofLemma}
\begin{lemma}\label{diffFsqrt}
For any $\epsilon > 0$ and $k\in\{0, 1, \ldots, \lfloor \sqrt{n}\rfloor\}$ almost surely
\begin{eqnarray}
|\sqrt{\mathbb{F}_n^{-1}(\frac{1}{2} + \frac{k}{n})} - \sqrt{\mathbb{F}_\gamma^{-1}(\frac{1}{2} + \frac{k}{n})}| \leq C_{\epsilon,\gamma} n^{-1/2 + \epsilon},
\end{eqnarray}
where $C_{\epsilon, \gamma}$ is a constant that only depends on $\gamma$ and $\epsilon$.
\end{lemma}

\begin{proof}
Let 
\begin{eqnarray*}
\Delta_{n, \gamma}^* := \sup_x |\mathbb{F}_n(x) - \mathbb{F}_\gamma(x)|.
\end{eqnarray*}
In~\cite{Gotze04}, G\"otze and Tikhomirov have shown that for any $\epsilon > 0$,  the rate of almost sure convergence of $\Delta_{n, \gamma}^*$ is at most $O(n^{-1/2 + \epsilon})$. This means that there exist a constant $C_\epsilon$ such that $\Delta_{n, \gamma}^*\leq C_\epsilon n^{-1/2 + \epsilon}$ almost surely. 

For $t = \mathbb{F}_{\gamma}^{-1}(\frac{1}{2} + \frac{k}{n})$ and $\delta > 0$ such that $[t - \delta, t + \delta]\subset (\gamma_{-}, \gamma_{+})$ we have
\begin{eqnarray*}
\mathbb{F}_\gamma(t + \delta) - \mathbb{F}_\gamma(t - \delta) &=& \int_{t - \delta}^{t + \delta} d\mu_{MP}^\gamma(x) \\
&=& \frac{1}{2\gamma\pi}\int_{t - \delta}^{t + \delta} \frac{\sqrt{(\gamma_{+} - x)(x - \gamma_{-})}}{x}dx \\
&\leq & C_\gamma \int_{t - \delta}^{t + \delta}\frac{1}{x}dx \\
&=& -C_\gamma\log(1 - \frac{2\delta}{t + \delta}) \\
&=& C_\gamma\delta.
\end{eqnarray*}
Note that 
\begin{eqnarray}
|\mathbb{F}_n(\mathbb{F}_n^{-1}(\frac{1}{2} + \frac{k}{n})) - \mathbb{F}_\gamma(\mathbb{F}_n^{-1}(\frac{1}{2} + \frac{k}{n}))| \leq \Delta_{n, \gamma}^* \leq  C_\epsilon n^{-1/2 + \epsilon},
\end{eqnarray}
where the second inequality is almost sure. Now if $|\mathbb{F}_n^{-1}(\frac{1}{2} + \frac{k}{n}) - \mathbb{F}_\gamma^{-1}(\frac{1}{2} + \frac{k}{n})|\geq \delta$ then we have
\begin{eqnarray}
|\mathbb{F}_\gamma(\mathbb{F}_n^{-1}(\frac{1}{2} + \frac{k}{n})) - \mathbb{F}_\gamma(\mathbb{F}_\gamma^{-1}(\frac{1}{2} + \frac{k}{n}))| \geq C_\gamma\delta
\end{eqnarray}
and this is possible only for $\delta\leq C_\epsilon n^{-1/2 + \epsilon}$. Therefore 
\begin{eqnarray}
|\mathbb{F}_n^{-1}(\frac{1}{2} + \frac{k}{n}) - \mathbb{F}_\gamma^{-1}(\frac{1}{2} + \frac{k}{n})| \leq C_{\epsilon, \gamma} n^{-1/2 + \epsilon}
\end{eqnarray}
almost surely. Using the mean value theorem for function $f(x) = \sqrt{x}$ we have
\begin{eqnarray}
|\sqrt{\mathbb{F}_n^{-1}(\frac{1}{2} + \frac{k}{n})} - \sqrt{\mathbb{F}_\gamma^{-1}(\frac{1}{2} + \frac{k}{n})}| \leq C_{\epsilon,\gamma} n^{-1/2 + \epsilon}.
\end{eqnarray}
\end{proof}

\begin{lemma}\label{fourthMomentSigmaHat}
Let $X = [x_{ij}]$ be a $m\times n$ random matrix with $x_{ij}$'s distributed i.i.d. from a sub-Gaussian distribution such that $\ee(x_{ij}) = m_{ij}$ and $Var(x_{ij}) = \sigma^2$ for some unknown value of $\sigma$. For any arbitrary $\epsilon > 0$  and 
\begin{eqnarray*}
\hat{\sigma} = \frac{\emph{\text{med}}(\lambda_i(X))}{\sqrt{n\mu_\gamma}},
\end{eqnarray*}
we have
\begin{eqnarray*}
\ee[\hat{\sigma}^4]\leq 4\sigma^4 + \frac{C_{\epsilon, \gamma}\sigma^4}{n^{2 - \epsilon}} + \frac{64\|M\|_*^4}{n^4\mu_\gamma^2},
\end{eqnarray*}
where $C_{\epsilon, \gamma} > 0$ is a constant independent of $n$ and $\sigma$.
\end{lemma}

\begin{proof}
Using inequality $(a + b)^2 \leq 2a^2 + 2b^2$ we write
\begin{eqnarray*}
\ee[\hat{\sigma}^4] &\leq & 8\ee[(\hat{\sigma} - \sigma)^4] + 8\sigma^4.
\end{eqnarray*}
By adding and subtracting $\lambda_{m/2}(X - M)$ we have
\begin{eqnarray*}
\ee[(\hat{\sigma} - \sigma)^4] &\leq & \frac{8}{n^2\mu_\gamma^2}\ee[(\lambda_{m/2}(X) - \lambda_{m/2}(X - M))^4] \\
&+& \frac{8\sigma^4}{n^2\mu_\gamma^2}\ee[(\lambda_{m/2}(\frac{X - M}{\sigma}) - \sqrt{n\mu_\gamma})^4].
\end{eqnarray*}
Note that by Lemma~\ref{diffFsqrt}
\begin{eqnarray}
\frac{8\sigma^4}{n^2\mu_\gamma^2}\ee\left[\left(\lambda_{m/2}(\frac{X - M}{\sigma}) - \sqrt{n\mu_\gamma}\right)^4\right] & = & \frac{8\sigma^4}{\mu_\gamma^2}\ee\left[\left(\sqrt{\mathbb{F}_n^{-1}(\frac{1}{2})} - \sqrt{\mathbb{F}_\gamma^{-1}(\frac{1}{2})}\right)^4\right] \nonumber\\
&\leq & \frac{\sigma^4 C_{\epsilon, \gamma}}{\mu_\gamma^2 n^{2 - \epsilon}}. \label{beforelast}
\end{eqnarray}
Using decomposition~\eqref{usefulDecomposition} we have
\begin{eqnarray*}
\lefteqn{\ee\left[(\lambda_{m/2}(X) - \lambda_{m/2}(X - M))^4\right] } \\
&\leq & 512 n^2 \sigma^4 \ee\left[\left(\sqrt{\mathbb{F}_n^{-1}(\frac{1}{2} - \frac{k}{n})} - \sqrt{\mathbb{F}_\gamma^{-1}(\frac{1}{2} - \frac{k}{n})}\right)^4\right] \\
& + & 512 n^2 \sigma^4 \ee\left[\left(\sqrt{\mathbb{F}_n^{-1}(\frac{1}{2} + \frac{k}{n})} - \sqrt{\mathbb{F}_\gamma^{-1}(\frac{1}{2} + \frac{k}{n})}\right)^4\right] \\
& + & 1028 n^2 \sigma^4 \ee\left[\left(\sqrt{\mathbb{F}_n^{-1}(\frac{1}{2})} - \sqrt{\mathbb{F}_\gamma^{-1}(\frac{1}{2})}\right)^4\right] \\
& + & 64 n^2\sigma^4 \left[\sqrt{\mathbb{F}_\gamma^{-1}(\frac{1}{2} + \frac{k}{n})} - \sqrt{\mathbb{F}_\gamma^{-1}(\frac{1}{2})} \right]^4 \\
& + & 64 n^2 \sigma^4\left[\sqrt{\mathbb{F}_\gamma^{-1}(\frac{1}{2} - \frac{k}{n})} - \sqrt{\mathbb{F}_\gamma^{-1}(\frac{1}{2})} \right]^4 \\
& + & 8(\frac{\|M\|_*}{k + 1})^4.
\end{eqnarray*}
Then Lemma~\ref{diffFsqrt} gives
\begin{eqnarray}\label{last}
\ee\left[(\lambda_{m/2}(\frac{X}{\sigma}) - \lambda_{m/2}(\frac{X - M}{\sigma}))^4\right] & \leq & C_{\epsilon, \gamma} n^\epsilon\sigma^4 + C_\gamma\sigma^4 + 8(\frac{\|M\|_*}{k + 1})^4.
\end{eqnarray}
Therefore~\eqref{last}, and~\eqref{beforelast} with $k = \lfloor\sqrt{n}\rfloor$ completes the proof,
\begin{eqnarray}
\ee[(\hat{\sigma} - \sigma)^4] &\leq &\frac{C_{\epsilon, \gamma}\sigma^4}{n^{2 - \epsilon}} + \frac{C_{\epsilon, \gamma}\sigma^4}{n^{2 - \epsilon}\mu_\gamma^2} + \frac{C_\gamma\sigma^4}{n^2\mu_\gamma^2} + \frac{64\|M\|_*^4}{n^4\mu_\gamma^2} \nonumber\\
&\leq & \frac{C_{\epsilon, \gamma}\sigma^4}{n^{2 - \epsilon}} + \frac{64\|M\|_*^4}{n^4\mu_\gamma^2}.
\end{eqnarray}
\end{proof}
\section*{Acknowledgement}
I am grateful to my advisor Sourav Chatterjee for his constant encouragement and insightful conversations and comments. I thank Matan Gavish and Amir Dembo for their helpful comments.

\end{document}